\documentclass[11pt]{amsart}

\usepackage{amsmath}
\usepackage{amssymb}
\usepackage{amsthm}
\usepackage{amscd}
\usepackage{mathrsfs}

\newcommand{\ri}{\mathfrak{o}}
\newcommand{\mi}{\mathfrak{p}}

\newcommand{\Z}{\mathbf{Z}}
\newcommand{\C}{\mathbf{C}}

\newcommand{\e}{{\varepsilon}}
\newcommand{\p}{\varpi}

\def\Section#1{\section{#1}\setcounter{equation}{0}}

\theoremstyle{plain}
\newtheorem{thm}[equation]{Theorem}
\newtheorem{lem}[equation]{Lemma}
\newtheorem{prop}[equation]{Proposition}
\newtheorem{cor}[equation]{Corollary}
\theoremstyle{definition}

\newtheorem{rem}[equation]{Remark}

\title{Whittaker functions associated to newforms  for $\mathrm{GL}(n)$ over $p$-adic fields}
\author{Michitaka Miyauchi}
\date{}
\address{
Department of Mathematics, Faculty of Science\\
Kyoto University\\
Oiwake Kita-Shirakawa Sakyo Kyoto 606-8502 JAPAN
}
\email{miyauchi@math.kyoto-u.ac.jp}

\begin{document}
%%%%%%%%%%%%%%%%%%%%%%%%%%%%%%%%%%

\begin{abstract}
Let $F$ be a non-archimedean local field of characteristic zero.
Jacquet, Piatetski-Shapiro and Shalika introduced the notion of 
newforms for irreducible generic representations of $\mathrm{GL}_n(F)$.
In this paper, we give an explicit formula for 
Whittaker functions associated to newforms 
on the diagonal matrices in $\mathrm{GL}_n(F)$.
\end{abstract}

\maketitle
\pagestyle{myheadings}
\markboth{}{}

\section{Introduction}
Let $F$ be a non-archimedean local field
of characteristic zero.
Shintani \cite{Shintani} gave
an explicit formula for spherical Whittaker functions
of unramified principal series representations
of $\mathrm{GL}_n(F)$.
His formula is a key to the unramified computation
of Rankin-Selberg type zeta integrals
(see for example \cite{B-F} and \cite{B-G}).
In this paper,
we extend Shintani's result to Whittaker functions associated to 
newforms for $\mathrm{GL}_n(F)$.

Jacquet, Piatetski-Shapiro and Shalika \cite{JPSS}
introduced the notion of newforms for irreducible generic representations of $\mathrm{GL}_n(F)$,
which is an extension of that for $\mathrm{GL}_2(F)$ by
Casselman \cite{Casselman}.
Newforms for $\mathrm{GL}_n(F)$
are defined by using a certain family of open compact subgroups
$\{K_n\}_{n \geq 0}$.
Given an irreducible generic representation $\pi$
of $\mathrm{GL}_n(F)$,
the smallest integer $c(\pi)$ among those
$n$ such that $\pi$ has $K_n$-fixed vectors
is called the conductor of $\pi$.
We say that a vector in $\pi$ is a newform
if it is fixed by $K_{c(\pi)}$.
When the conductor of $\pi$ is zero,
its newforms are just $\mathrm{GL}_n(\ri)$-fixed vectors,
where $\ri$ is the ring of integers in $F$.
In this paper, 
we give an explicit formula 
for Whittaker functions associated to newforms
on the diagonal matrices in $\mathrm{GL}_n(F)$.

Originally,
Shintani's explicit formula for spherical Whittaker functions
is written in terms of Hecke eigenvalues.
We will follow his method.
For an irreducible generic representation $\pi$,
the Hecke algebra associated to $K_{c(\pi)}$
acts on the space of its newforms.
Since this space is one-dimensional,
the actions of elements in this  Hecke algebra are 
given by scalar multiplication.
Suppose that the conductor of $\pi$ is positive.
Then, similar to the unramified case,
we get a formula for the Whittaker function $W$
associated to a newform 
on $T_1=\{\mathrm{diag}(a_1,\ldots, a_{n-1},1)\, |\, a_i \in F^\times\}$
in terms of Hecke eigenvalues $\lambda_1, \ldots, \lambda_{n-1}$
(see section~\ref{sec:hecke} for precise definition).
Kondo and Yasuda \cite{K-S} showed the relation
between these Hecke eigenvalues and the $L$-factor
of $\pi$.
We therefore obtain
an explicit formula 
for $W$ on $T_1$
in terms of 
the $L$-factor of $\pi$ (Theorem~\ref{thm:main}).

As a corollary,
we show that 
Whittaker functions associated to newforms attain 
$L$-factors when they are integrated on $\mathrm{GL}_1(F)$
which is 
embedded into the upper left side in $\mathrm{GL}_n(F)$ (Theorem~\ref{thm:zl}).

We note that 
our formula determines Whittaker functions 
associated to newforms 
for $\pi$
on $BK_{c(\pi)}$, where $B$ denotes the upper triangular Borel 
subgroup of $\mathrm{GL}_n(F)$.
But the set $BK_{c(\pi)}$ is smaller than $\mathrm{GL}_n(F)$ 
when the conductor of $\pi$ is positive.
However
it seems that 
our formula is enough to compute several kinds of
zeta integrals when all data in them are related to 
newforms.

Recently, Matringe \cite{Matringe} gave a constructive proof of the existence
of newforms for generic representations $\pi$ of $\mathrm{GL}_n(F)$,
by investigating derivatives of $\pi$.
One can find
the same formula for Whittaker functions associated to newforms in {\it loc. cit}.

\medskip
\noindent
{\bf Acknowledgements} \
The author would like to thank Satoshi Kondo,
Takuya Yamauchi and Seidai Yasuda
for helpful discussions and comments.
He also thanks the referee for useful suggestions. 

\Section{Local newforms}\label{sec:newform}
In this section,
we recall  from \cite{JPSS} the notion  and basic properties 
of local newforms for $\mathrm{GL}(n)$.
Let 
$F$ be a non-archimedean local field of characteristic zero,
$\ri$ its ring of integers,
$\mi$ the maximal ideal in $\ri$,
and
$\p$ a generator of $\mi$.
We write $|\cdot|$  for the absolute value of $F$
normalized so that $|\p| = q^{-1}$,
where 
$q$ denotes the cardinality of the residue field $\ri/\mi$.
We fix a non-trivial additive character $\psi$ of $F$
whose conductor is $\ri$.

We 
set
$G = \mathrm{GL}_n(F)$.
Let
$B$ denote the Borel subgroup of $G$ consisting of the upper triangular 
matrices,
and $U$ its unipotent radical.
We use the same letter 
$\psi$ for the following character of $U$ induced from $\psi$:
\[
\psi(u) = \psi(\sum_{i=1}^{n-1}u_{i,i+1}),\
\mathrm{for}\ u = (u_{i,j}) \in U.
\]
Let  $(\pi, V)$ be
an irreducible generic representation of $G$.
Then there exists a unique element $l$ in $\mathrm{Hom}_U(\pi, \psi)$ up to constant.
For $v \in V$,
we define the Whittaker function  associated to $v$
by
\[
W_v(g) = l(\pi(g)v),\ g \in G.
\]
We call 
$\mathcal{W}(\pi, \psi) = \{ W_v\, |\, v \in V\}$ 
{\it the  Whittaker model of $\pi$}
with respect to $\psi$.

Put $K_0 = \mathrm{GL}_n(\ri)$.
For each positive integer $m$,
let $K_m$ be the subgroup of $K_0$
consisting of the elements $k = (k_{ij})$ in $K_0$
which satisfy
\[
(k_{n1}, k_{n2}, \ldots, k_{nn}) \equiv (0, 0, \ldots, 0, 1) \pmod{\mi^m}.
\]
We write $V(m)$ for the space of $K_m$-fixed vectors in $V$.
Due to \cite{JPSS} (5.1) Th\'{e}or\`{e}me (ii),
there exists a non-negative integer 
$m$ such that
$V(m)\neq \{0\}$.
We denote by $c(\pi)$ the smallest integer with this property.
We 
call $c(\pi)$
{\it the conductor of $\pi$},
and elements in $V(c(\pi))$ {\it newforms for $\pi$}.
By \cite{JPSS} (5.1) Th\'{e}or\`{e}me (ii),
we have
\begin{eqnarray}\label{eq:mult_one}
\dim V(c(\pi)) = 1.
\end{eqnarray}

For simplicity,
we say that an element $W$ in $\mathcal{W}(\pi, \psi)$
is a newform 
if $W$ is the Whittaker function associated to a newform for $\pi$.
Suppose that
$W$ is a non-zero newform in $\mathcal{W}(\pi, \psi)$.
Then by the existence of Kirillov model for $\pi$
(see \cite{BZ} Theorem 5.20),
there exists an element $g \in \mathrm{GL}_{n-1}(F)$
such that
\[
W\left(
\begin{array}{cc}
g & \\
& 1
\end{array}
\right) \neq 0.
\]
For any element
$f = (f_1, \ldots, f_{n-1})$ in $\Z^{n-1}$,
we set
\[
\p^f
= \mathrm{diag} (\p^{f_1}, \ldots, \p^{f_{n-1}}, 1)
\in G.
\]
By using the Iwasawa decomposition of 
$\mathrm{GL}_{n-1}(F)$,
we see that there exists $f \in \Z^{n-1}$
such that
$W(\p^f) \neq 0$.
%%%
\begin{prop}\label{prop:ij}
Let
$W$ be a newform in $\mathcal{W}(\pi, \psi)$.
For $f \in \Z^{n-1}$,
we have
$W(\p^f) = 0$  unless
$f_1 \geq \ldots \geq f_{n-1} \geq 0$.
\end{prop}
%%%%%
\begin{proof}
The proposition follows since
$W$ is $U\cap M_n(\ri)$-invariant.
\end{proof}

%%%%
%%%%
%%%%
\Section{Hecke operators}\label{sec:hecke}
Let 
$(\pi, V)$ be an irreducible generic representation of $G$
with conductor $c = c(\pi)$.
Throughout this section,
we suppose that $c$ is positive.
For each 
$g \in G$,
we define the
Hecke operator $T_g$ on $V(c)$
by
\begin{eqnarray*}
T_g v & = & 
\frac{1}{\mathrm{vol}(K_c)} \int_{K_cgK_c} \pi(k) v dk
 = 
\sum_{k \in K_c/K_c\cap gK_cg^{-1}}\pi(kg)v,
\end{eqnarray*}
for $v \in V(c)$.
By (\ref{eq:mult_one}),
there exists a complex number $\lambda_g$
such that
$T_g = \lambda_g 1_{V(c)}$.
We call
$\lambda_g$ the Hecke eigenvalue of $T_g$.

For
$1 \leq i \leq n-1$,
we denote by $\lambda_i$ the Hecke eigenvalue
of $T_i = T_{\p^{f^i}}$,
where
\[
f^i = (\overbrace{1,\ldots, 1}^i, 0, \ldots, 0) \in \Z^{n-1}.
\]
To
describe $T_i$,
we give
a complete system of representatives 
for $K_m/K_m\cap \p^{f^i} K_m \p^{-f^i}$
($m > 0$, $1 \leq i \leq n-1$).
We write $A \in M_n(F)$
as
\[
A = 
\left(
\begin{array}{cc}
A_{11} & A_{12}\\
A_{21} & A_{22}
\end{array}
\right),
\]
where $A_{11} \in M_{n-1}(F)$,
$A_{12} \in M_{n-1, 1}(F)$, 
$A_{21} \in M_{1, n-1}(F)$
and $A_{22} \in F$.
We embed
the group $\mathrm{GL}_{n-1}(F)$ into $G$
by 
\[
g \mapsto 
\left(
\begin{array}{cc}
g & \\
& 1
\end{array}
\right),\ g \in \mathrm{GL}_{n-1}(F).
\]
Then
we may regard $\p^f$,
$f \in \Z^{n-1}$
as an element in $\mathrm{GL}_{n-1}(F)$.
Set $H = \mathrm{GL}_{n-1}(\ri)$.
For $0 \leq i \leq n-1$,
we define an $\ri$-lattice $L_i$ in $M_{n-1, 1}(F)$
by
\[
L_i = {}^t (\overbrace{\mi\oplus \cdots \oplus \mi}^i
\oplus \ri \oplus \cdots \oplus \ri).
\]
Then the groups $H$ and
$H\cap \p^{f^i} H \p^{-f^i}$
fix
$L_0$ and $L_i$ respectively.
%%%
\begin{lem}\label{lem:coset}
Let $m$ be a positive integer.
For $1 \leq i \leq n-1$,
we can take 
\[
\left(
\begin{array}{cc}
a & x\\
0 & 1
\end{array}
\right),\ a \in H/H\cap \p^{f^i} H \p^{-f^i},\ x \in L_0/aL_i
\]
as a complete system of representatives 
for $K_m/K_m\cap \p^{f^i} K_m \p^{-f^i}$.
\end{lem}
%%%
\begin{proof}
We use the block notation.
An element $g=\left(
\begin{array}{cc}
a & b\\
c & d
\end{array}
\right) \in K_m$ lies in 
$K_m \cap \p^{f^i}K_m\p^{-f^i}$
if and only if 
$a \in H\cap \p^{f^i}H\p^{-f^i}$
and $b \in L_i$.
Thus, one can observe that the elements in the lemma
belong to pairwise disjoint cosets in 
$K_m/K_m\cap \p^{f^i} K_m \p^{-f^i}$.
For 
$g=\left(
\begin{array}{cc}
a & b\\
c & d
\end{array}
\right) \in K_m$,
we see that $g$ is equivalent
to 
$\left(
\begin{array}{cc}
a & b\\
0 & 1
\end{array}
\right)$ modulo $K_m\cap \p^{f^i} K_m \p^{-f^i}$.
This completes  the proof of the lemma.
\end{proof}

Let $W$ be a newform in $\mathcal{W}(\pi, \psi)$.
Set $w(f) = W(\p^f)$, for $f \in \Z^{n-1}$.
Then we obtain the following lemma:
%%%
\begin{lem}\label{lem:diff}
Suppose that 
$f \in \Z^{n-1}$ satisfies 
$f_1 \geq \ldots \geq f_{n-1} \geq 0$.
Then we have
\begin{eqnarray}
q^{-i} \lambda_i w(f)
=
q^{i(n-1)-i(i-1)/2}
\sum_{\e \in I_i} q^{ -\sum_{j=1}^{n-1} \e_j j} w(f+\e),
\end{eqnarray}
where 
$I_i=
\{ \e \in \Z^{n-1}\ |\ \e_j \in \{0, 1\},\
\sum_{j = 1}^{n-1} \e_j = i\}$.
\end{lem}
%%%%
\begin{proof}
By Lemma~\ref{lem:coset},
we get
\begin{eqnarray*}
\lambda_i W(\p^f) & = & 
\sum_{k \in K_c/K_c\cap \p^{f^i}K_c\p^{-f^i}}W(\p^f k \p^{f^i})\\
& = & 
\sum_{a \in H/H\cap \p^{f^i} H \p^{-f^i},\ x \in L_0/aL_i}W\left(\p^f
\left(
\begin{array}{cc}
1 & x\\
0 & 1
\end{array}
\right)
\left(
\begin{array}{cc}
a & 0\\
0 & 1
\end{array}
\right)\p^{f^i}
\right)\\
& = & 
q^i
\sum_{a \in H/H\cap \p^{f^i} H \p^{-f^i}}W\left(\p^f
\left(
\begin{array}{cc}
a & 0\\
0 & 1
\end{array}
\right)\p^{f^i}
\right).
\end{eqnarray*}
In the last equality,
we use the equation
$[L_0: aL_i] = [aL_0:aL_i] = [L_0:L_i] = q^i$.
Now the proof is quite similar to that
of the theorem in \cite{Shintani} p. 181 
because $W|_{\mathrm{GL}_{n-1}(F)}$
is $\mathrm{GL}_{n-1}(\ri)$-invariant.
We note that 
in \cite{Shintani}
one should take the set $I_i$ as
$I_i=
\{ \e \in \Z^{n}\ |\ \e_j \in \{0, 1\},\
\sum_{j = 1}^{n} \e_j = i\}$.
\end{proof}

%%%%%
%%%%%
%%%%%
\Section{An explicit formula for Whittaker functions}
We prepare some notation to state our main theorem.
For an 
irreducible generic representation $\pi$ of $G$,
let $L(s, \pi)$ denote its $L$-factor 
defined in \cite{GJ}.
It follows from \cite{Jacquet2} section 3
that
the degree of $L(s, \pi)$ is
equal to or less than $n$.
So we can write it as
\[
L(s, \pi) =\prod_{i = 1}^n (1-\alpha_i q^{-s})^{-1},\
\alpha_i \in \C.
\]
By \cite{Jacquet2} section 3 again,
$L(s, \pi)$ is of degree $n$
if and only if $\pi$ is unramified,
that is, $c(\pi)$ equals to zero.
We take $\alpha_n$ to be zero 
if the degree of $L(s, \pi)$ is less than $n$.
Let
$X = (X_1, \ldots, X_{n})$ be $n$ indeterminates.
If $f \in \Z^{n}$ 
satisfies $f_1 \geq \ldots \geq f_n \geq 0$,
then
we denote by $s_f(X)$ the Schur polynomial 
in $X_1, \ldots, X_{n}$
associated to $f$,
that is,
\begin{eqnarray*}
s_f(X) = \frac{|(X_j^{f_i+n-i})_{1 \leq i, j\leq n}|}{\prod_{1 \leq i < j \leq n}(X_i-X_j)}
\end{eqnarray*}
(see \cite{Mac} Chapter I, section 3).
Since $s_f(X)$ is a symmetric polynomial
in $X_1, \ldots, X_n$,
the number $s_f(\alpha) = s_f(\alpha_1, \ldots, \alpha_n)$
is well-defined.
We identify $(f_1, \ldots, f_{n-1}) \in \Z^{n-1}$
with $(f_1, \ldots, f_{n-1}, 0) \in \Z^n$.
We note that if the conductor of $\pi$ is positive,
then 
we have
$s_f(\alpha) = s_{(f_1, \ldots, f_{n-1}, 0)}(\alpha_1, \ldots, \alpha_{n-1}, 0)=
s_{(f_1, \ldots, f_{n-1})}(\alpha_1, \ldots, \alpha_{n-1})$,
for $f \in \Z^{n-1}$ such that 
$f_1 \geq \ldots \geq f_{n-1}\geq 0$.

We denote by $\delta_B$ the modulus character of $B$.
We have
$\delta_B(\p^f) = q^{-\sum_{j=1}^{n-1}({n+1}-2j)f_j}$,
for $f \in \Z^{n-1}$.
%%%
\begin{thm}\label{thm:main}
Let $\pi$ be an irreducible generic representation of $G$
and $W$ its newform in $\mathcal{W}(\pi, \psi)$.
For $f \in \Z^{n-1}$,
we have
\begin{eqnarray*}
W(\p^f)
= 
\left\{
\begin{array}{cl}
\delta_B^{1/2}(\p^f) s_f(\alpha)W(1), & \mathrm{if}\ 
f_1 \geq \ldots \geq f_{n-1} \geq 0;\\
0, & \mathrm{otherwise}.
\end{array}
\right.
\end{eqnarray*}
\end{thm}

%%%%
\begin{proof}
If $c(\pi) = 0$,
then the theorem follows from \cite{Shintani}.
So we may assume that $\pi$ has positive conductor.
For $f \in \Z^{n-1}$,
we set
\begin{eqnarray*}
\widetilde{w}(f)
=
q^{\sum_{j = 1}^{n-1}(n-1-j)f_j} W(\p^f).
\end{eqnarray*}
By Proposition~\ref{prop:ij} and Lemma~\ref{lem:diff},
the function
$\widetilde{w}$ on $\Z^{n-1}$
satisfies the following system of difference equations:
\begin{eqnarray}\label{eq:diff_eq}
\left\{
\begin{array}{cl}
q^{i(i-1)/2-i}\lambda_i \widetilde{w}(f)
=
\sum_{\e \in I_i} \widetilde{w}(f+\e), & \mathrm{if}\
f_1 \geq \ldots \geq f_{n-1} \geq 0;\\
\widetilde{w}(f) = 0, & 
\mathrm{otherwise}.
\end{array}
\right.
\end{eqnarray}
As in \cite{Shintani} p. 182,
the solution of the above difference equations
is unique
and given by
\begin{eqnarray*}
\widetilde{w}(f)
= 
\left\{
\begin{array}{cl}
s_f(\mu)W(1), & \mathrm{if}\ 
f_1 \geq \ldots \geq f_{n-1} \geq 0;\\
0, & \mathrm{otherwise},
\end{array}
\right.
\end{eqnarray*}
where 
$\mu_1, \ldots, \mu_{n-1}$ are complex numbers whose
$i$-th elementary symmetric polynomial 
equals to 
$q^{i(i-1)/2-i}\lambda_i$, for $1 \leq i \leq n-1$,
and $\mu_n = 0$.

By \cite{K-S} Theorem 4.2,
we have
\[
L(s, \pi)
= \left(
\sum_{i = 0}^{n-1} (-1)^i \lambda_i q^{\frac{i(i-1)}{2}-i(\frac{n-1}{2}+s)}
\right)^{-1}.
\]
Hence we may assume $\mu_i =q^{(n-1)/2-1}\alpha_i$,
for $1 \leq i \leq n$.
Thus, if $f \in \Z^{n-1}$ satisfies
$f_1 \geq \ldots \geq f_{n-1} \geq 0$,
then
we obtain
%%%%%
\begin{eqnarray*}
W(\p^f) & = & 
q^{-\sum_{j = 1}^{n-1}(n-1-j)f_j} 
s_f(\mu)W(1)
= 
q^{-\sum_{j = 1}^{n-1}(n-1-j)f_j +(\frac{n-1}{2}-1)\sum_{j=1}^{n-1}f_j} 
s_f(\alpha)W(1)\\
&= & 
q^{-\sum_{j=1}^{n-1}(\frac{n+1}{2}-j)f_j} 
s_f(\alpha)W(1)
=
\delta_B^{1/2}(\p^f) s_f(\alpha)W(1).
\end{eqnarray*}
This completes the proof.
\end{proof}

%%%
\begin{rem}
Set $D_1 = \{ \p^f\, |\, f \in \Z^{n-1}\}$.
Since the center $Z$ of $G$ acts on $\mathcal{W}(\pi, \psi)$
by the central character of $\pi$,
Theorem~\ref{thm:main} gives an explicit formula
for newforms in $\mathcal{W}(\pi, \psi)$ on $BK_{c(\pi)} = UZD_1 K_{c(\pi)}$.
\end{rem}

%%%%
\begin{cor}
Let $\pi$ be an irreducible generic representation of $G$.
Then we have $W(1) \neq 0$
 for all non-zero newforms $W$ in $\mathcal{W}(\pi, \psi)$.
\end{cor}
%%%
\begin{proof}
By Theorem~\ref{thm:main},
$W(1) = 0$ implies 
$W(\p^f) = 0$ for all $f \in \Z^{n-1}$.
This contradicts the remark before
Proposition~\ref{prop:ij}.
\end{proof}

\Section{An application to zeta integral}
In this section,
we give an integral representation of $L$-factors
by using our formula for Whittaker functions associated to 
newforms.
%%%%%
%%%%%
Let 
$(\pi, V)$ be an irreducible generic representation of $G$.
For 
$W\in \mathcal{W}(\pi, \psi)$,
we set 
\[
Z(s, W)
=
\int_{F^\times}
W(t(a))
|a|^{s-\frac{n-1}{2}}
d^\times a,\
s \in \C,
\]
where 
$t(a) = \mathrm{diag}(a, 1, \ldots, 1)$, for $a \in F^\times$.
Here we normalize Haar measure $d^\times a$ on $F^\times$
so that
$\int_{\ri^\times} d^\times a = 1$.
The integral $Z(s, W)$
absolutely converges to a rational function in $q^{-s}$
when the real part of $s$
is sufficiently large.
By \cite{Rankin} Theorem 2.7 (ii),
the set
$\{ Z(s, W)\, |\, W \in \mathcal{W}(\pi, \psi)\}$ 
coincides with the fractional 
ideal of $\C[q^{-s}, q^s]$ 
generated by $L(s, \pi)$.
We shall show that 
$Z(s, W)$ attains $L(s, \pi)$ when $W$ is a newform.
%%%
\begin{thm}\label{thm:zl}
Let $\pi$ be an irreducible generic representation of $G$
and 
$W$ the newform in $\mathcal{W}(\pi, \psi)$
such that $W(1)= 1$.
Then we have
\[
Z(s, W) = L(s, \pi).
\]
\end{thm}
%%%%%
\begin{proof}
By Proposition~\ref{prop:ij}
and Theorem~\ref{thm:main},
we obtain
\begin{eqnarray*}
Z(s, W) & = & \sum_{k =0}^\infty W(t(\p^k)) |\p^k|^{s-\frac{n-1}{2}}
=
\sum_{k =0}^\infty \delta_B^{1/2}(t(\p^k))s_{(k, 0, \ldots, 0)}(\alpha) |\p^k|^{s-\frac{n-1}{2}}\\
&=&
\sum_{k =0}^\infty s_{(k, 0, \ldots, 0)}(\alpha) q^{-ks}.
\end{eqnarray*}
It follows from \cite{Mac} Chapter I (3.4) that
$s_{(k, 0, \ldots, 0)}(\alpha)$ 
is the complete homogeneous symmetric polynomial of
degree $k$,
that is,
\[
s_{(k, 0, \ldots, 0)}(\alpha)
= \sum_{k_1+\cdots+k_n = k} \alpha_1^{k_1}\cdots \alpha_n^{k_n}.
\]
Hence we get
\begin{eqnarray*}
Z(s, W) & = & 
\sum_{k =0}^\infty \left(\sum_{k_1+\cdots+k_n = k} \alpha_1^{k_1}\cdots \alpha_n^{k_n}\right) q^{-ks}
=
\prod_{i = 1}^n (\sum_{k_i= 0}^\infty \alpha_i^{k_i}q^{-k_is})\\
& = & 
\prod_{i= 1}^n (1-\alpha_iq^{-s})^{-1}
= L(s, \pi),
\end{eqnarray*}
as required.
\end{proof}


\begin{thebibliography}{10}

\bibitem{BZ}
I.~N. Bern{\v{s}}te{\u\i}n and A.~V. Zelevinski{\u\i}.
\newblock Representations of the group {$GL(n,F),$} where {$F$} is a local
  non-{A}rchimedean field.
\newblock {\em Uspehi Mat. Nauk}, 31(3(189)):5--70, 1976.

\bibitem{B-F}
D.~Bump and S.~Friedberg.
\newblock The exterior square automorphic {$L$}-functions on {${\rm GL}(n)$}.
\newblock In {\em Festschrift in honor of {I}. {I}. {P}iatetski-{S}hapiro on
  the occasion of his sixtieth birthday, {P}art {II} ({R}amat {A}viv, 1989)},
  volume~3 of {\em Israel Math. Conf. Proc.}, pages 47--65. Weizmann,
  Jerusalem, 1990.

\bibitem{B-G}
D.~Bump and D.~Ginzburg.
\newblock Symmetric square {$L$}-functions on {${\rm GL}(r)$}.
\newblock {\em Ann. of Math. (2)}, 136(1):137--205, 1992.

\bibitem{Casselman}
W.~Casselman.
\newblock On some results of {A}tkin and {L}ehner.
\newblock {\em Math. Ann.}, 201:301--314, 1973.

\bibitem{GJ}
R.~Godement and H.~Jacquet.
\newblock {\em Zeta functions of simple algebras}.
\newblock Lecture Notes in Mathematics, Vol. 260. Springer-Verlag, Berlin,
  1972.

\bibitem{Jacquet2}
H.~Jacquet.
\newblock Principal {$L$}-functions of the linear group.
\newblock In {\em Automorphic forms, representations and {$L$}-functions
  ({P}roc. {S}ympos. {P}ure {M}ath., {O}regon {S}tate {U}niv., {C}orvallis,
  {O}re., 1977), {P}art 2}, Proc. Sympos. Pure Math., XXXIII, pages 63--86.
  Amer. Math. Soc., Providence, R.I., 1979.

\bibitem{JPSS}
H.~Jacquet, I.~Piatetski-Shapiro, and J.~Shalika.
\newblock Conducteur des repr\'esentations du groupe lin\'eaire.
\newblock {\em Math. Ann.}, 256(2):199--214, 1981.

\bibitem{Rankin}
H.~Jacquet, I.~I. Piatetskii-Shapiro, and J.~A. Shalika.
\newblock Rankin-{S}elberg convolutions.
\newblock {\em Amer. J. Math.}, 105(2):367--464, 1983.

\bibitem{K-S}
S.~Kondo and S.~Yasuda.
\newblock Local {L} and epsilon factors in {H}ecke eigenvalues, preprint,
  {I}{P}{M}{U} preprints, {I}{P}{M}{U}10-0107, 2010.
\newblock available at \verb|http://db.ipmu.jp/ipmu/ipmuno/ipmu_list/|.

\bibitem{Mac}
I.~G. Macdonald.
\newblock {\em Symmetric functions and {H}all polynomials}.
\newblock Oxford Mathematical Monographs. The Clarendon Press Oxford University
  Press, New York, second edition, 1995.
\newblock With contributions by A. Zelevinsky, Oxford Science Publications.

\bibitem{Matringe}
N.~Matringe.
\newblock Essential {W}hittaker functions for {$\mathrm{GL}(n)$} over a
  {$p$}-adic field, preprint, ar{X}iv:1201.5506, 2012.

\bibitem{Shintani}
T.~Shintani.
\newblock On an explicit formula for class-{$1$} ``{W}hittaker functions'' on
  {$GL_{n}$} over {$\mathfrak{P}$}-adic fields.
\newblock {\em Proc. Japan Acad.}, 52(4):180--182, 1976.

\end{thebibliography}
\end{document}